\documentclass[12pt]{amsart}

\usepackage{color}
\usepackage{amsmath,amssymb,leftidx,fixltx2e, eucal,mathrsfs}

\usepackage{tikz-cd}
\usepackage{lipsum}
\usepackage{adjustbox}
\usepackage{verbatim}
\usepackage{graphicx}
\usepackage{enumerate}
\usepackage{mathtools, caption} 
\usepackage[cmtip,all]{xy}
\usepackage{upref}

\usepackage{hyperref}

\usepackage{tikz-cd}

\tikzset{node distance=2cm, auto}

\DeclareMathOperator*{\spn}{span}

\DeclareMathOperator*{\clspn}{\overline{\spn}}

\theoremstyle{plain}
\newtheorem{theorem}{Theorem}[section]
\newtheorem{prop}[theorem]{Proposition}
\newtheorem{corollary}[theorem]{Corollary}
\newtheorem{lemma}[theorem]{Lemma}
\theoremstyle{definition}
\newtheorem{definition}[theorem]{Definition}

\newtheorem{remark}[theorem]{Remark}

\numberwithin{equation}{section}

\newcommand{\secref}[1]{Section~\textup{\ref{#1}}}
\newcommand{\theoremref}[1]{Theorem~\textup{\ref{#1}}}
\newcommand{\corref}[1]{Corollary~\textup{\ref{#1}}}

\newcommand{\propref}[1]{Proposition~\textup{\ref{#1}}}
\newcommand{\defnref}[1]{Definition~\textup{\ref{#1}}}
\newcommand{\remarkref}[1]{Remark~\textup{\ref{#1}}}

\newcommand{\la}{\langle}
\newcommand{\ra}{\rangle}
\newcommand{\<}{\langle}
\renewcommand{\>}{\rangle}
\newcommand{\pre}[1]{{}_{#1}}
\newcommand{\variso}{\xrightarrow{\simeq}}

\renewcommand{\)}{\textup)}

\newcommand{\xt}{\otimes}
\renewcommand{\ker}{\operatorname{Ker}}
\newcommand{\aut}{\operatorname{Aut}}
\newcommand{\id}{\operatorname{id}}
\newcommand{\ad}{\operatorname{Ad}}

\newcommand{\C}{$C^*$}
\newcommand{\AXB}{$\pre AX_B$}

\newcommand{\otss}{\otimes_B}

\newcommand{\otsc}{\otimes_C}
\newcommand{\X}{\widetilde{X}}
\newcommand{\wilde}{\widetilde}

\newcommand{\KK}{\mathcal K}
\newcommand{\LL}{\mathcal L}
\newcommand{\CC}{\mathbb C}

\newcommand{\csta}{\ensuremath{C^*}-algebra}

\newcommand{\xm}{\otimes_{\max}}

\newcommand{\cst}{\ensuremath{C^*}}
\newcommand{\od}{\odot}

\newcommand{\ibm}{imprimitivity bimodule}
\newcommand{\hbm}{Hilbert bimodule}
\newcommand{\righttext}[1]{\qquad\text{#1 }}
\newcommand{\midtext}[1]{\qquad\text{#1 }\qquad}

\title{Exact sequences in the enchilada category}

\author[Eryuzlu]{M. Ery\"uzl\"u}
\address{School of Mathematical and Statistical Sciences
\\Arizona State University
\\Tempe, Arizona 85287}
\email{meryuzlu@asu.edu}

\author[Kaliszewski]{S. Kaliszewski}
\address{School of Mathematical and Statistical Sciences
\\Arizona State University
\\Tempe, Arizona 85287}
\email{kaliszewski@asu.edu}

\author[Quigg]{John Quigg}
\address{School of Mathematical and Statistical Sciences
\\Arizona State University
\\Tempe, Arizona 85287}
\email{quigg@asu.edu}

\date{July~21 , 2020}
\subjclass[2010]{Primary 46L55; Secondary 18B99}

\keywords{short exact sequence, $C^*$-correspondence, exact functor, crossed product}

\begin{document}
\maketitle
\begin{abstract}
We define exact sequences in the enchilada category of $C^*$-algebras and correspondences, and
prove that 
the 
reduced-crossed-product functor
is 
not
exact
for the enchilada categories.
Our motivation was
to determine whether we can have a better understanding of the Baum-Connes conjecture by using enchilada categories.
Along the way we prove numerous results showing that the enchilada category is rather strange.
\end{abstract}

\section{Introduction} 

The Baum-Connes conjecture says (very roughly) that, given an action of a locally compact group on a $C^*$-algebra, the topological $K$-theory is naturally isomorphic to the $K$-theory of the reduced crossed product. Unfortunately, the conjecture is false in that form, essentially because the topological $K$-theory is an exact functor of actions, while the reduced crossed product is not.
Some effort has been expended to ``fix'' the Baum-Connes conjecture (see, e.g., \cite{bgwexact, bew, bew2, bewminimal, klqfunctor, klqfunctor2, klqtensoraction}).
In this paper we investigate another possible strategy of fixing the conjecture: change the categories.
All the work to date on the {Baum-Connes} conjecture has used categories of $C^*$-algebras, possibly with extra structure, where the morphisms are *-homomorphisms that preserve the structure.
Here we change the morphisms to be isomorphism classes of $C^*$-correspondences
--- we call these ``enchilada categories''.
Perhaps we should explain the genesis of
this unusual name: when the AMS~Memoir \cite{enchilada} was being prepared,
the authors decided to first introduce the general idea by writing a smaller paper~\cite{taco},
and we privately referred to these two papers as the ``big enchilada'' and the ``little taco'', respectively.
Since then, a few of us have been using the name ``enchilada category'' for the type of category studied in those two papers (see Section~\ref{enchilada} for definitions).

More precisely, we investigate the following question:
is the reduced crossed product functor exact between enchilada categories?
In order to study this question rigorously, we first need to know:
what are the exact sequences are in the enchilada category of $C^*$-algebras?
In this paper we give one answer to this latter question.
We then apply this to answer the exactness question;
unfortunately, the answer is no, the reduced crossed product is not exact for the enchilada categories.

Despite this failure of exactness, we believe that our investigation into exact sequences in the enchilada category will be useful.
It turns out that the enchilada category is quite strange, in the sense that the morphisms are not mappings; additionally, the category is not abelian, or even preadditive, so the standard techniques of homological algebra are largely unavailable.
As an illustration of our ignorance concerning the enchilada category, we have been unable to completely characterize either the monomorphisms or the epimorphisms.

This contributed to our most formidable hurdle: how to define the image of a morphism.
Since a morphism in the enchilada category is (the isomorphism class of) a $C^*$-correspondence,
it is fairly easy to guess that the image should involve the range of the inner product, and it is then a short step to imagine that the range of the correspondence should in fact be the closed span of the inner products.
But how to put this on a rigorous footing?
In abstract category theory, a common way to define image uses \emph{subobjects}, which in turn use monomorphisms; this gave us trouble due to our inability to characterize monomorphisms.
In some category-theory literature, the definition of subobject is modified by restricting the class of monomorphisms.
We first tried the \emph{strong monomorphisms} of \cite{arduini}.
This in turn lead to another stumbling block: our limited understanding of epimorphisms in the enchilada category eventually defeated us because the definition of strong monomorphism uses epimorphisms.
We then tried using \emph{split monomorphisms} in the definition of image.
This turned out to work very well, but it was unsatisfying since it seemed to depend upon the more-or-less arbitrary choice of split monomorphisms.
Fortunately, we found in \cite{schubert} 
an alternative notion of image, which we call \emph{Schubert image},
based upon kernels and cokernels.
Since we were able to prove that the enchilada category has kernels and cokernels,
and even more importantly that in the enchilada category every kernel is a split monomorphism
(see \propref{split mono is left-full} and \corref{left-full is kernel}),
we were happy to adopt Schubert's definition of image.

We begin in \secref{prelim} with a brief review of the basic notions from category theory that we will need.
Then in \secref{enchilada} we investigate these notions for the enchilada category,
where our main objective is to define \emph{kernel} and \emph{Schubert image}.
Once this is done, we characterize short exact sequences in the enchilada category (see \theoremref{short exact}).
It is then easy to explain why the reduced crossed product functor is not exact for the enchilada categories (see \remarkref{failure}).
Finally, in \secref{further} we exhibit a few other ways in which the enchilada category is strange ---
it is not abelian, or even preadditive.

\section{Preliminaries}\label{prelim}

This paper is written primarily for $C^*$-algebraists.
We use a nontrivial portion of the concepts of category theory, so for the convenience of the reader we recall the definitions and basic results here.
All abstract discussions of morphisms and objects will be tacitly in some category~$\mathcal{C}$.

\begin{definition}
A \emph{monomorphism} is a morphism $f$
such that $f\circ g=f\circ h$ implies $g=h$.
Dually,  $f$ is an \emph{epimorphism} if $g\circ f=h\circ f$ implies $g=h$.
\end{definition}

\begin{definition}
An object $A$ is \emph{initial} if for every object $B$ there is exactly one morphism $A\to B$.
Dually, $A$ is \emph{terminal} if for every object $B$ there is exactly one morphism $B\to A$. Both of them are unique up to isomorphism.
\end{definition}

\begin{definition}
A \emph{zero} 
is an object that is both initial and terminal.
If it exists, a zero is unique up to isomorphism, and is denoted by 0.
For any two objects 
$A,B$
the \emph{zero morphism} 
$0_{A, B} :A\to B$
is the unique morphism from $A$ to $B$ that factors through 0.
Frequently we just write 0 for $0_{A,B}$.
\end{definition}

\begin{definition}
Let $f,g:A\to B$.
An \emph{equalizer} of $f,g$ is a morphism $h:C\to A$ such that
\begin{itemize}
\item
$f\circ h=g\circ h$;

\item
whenever $k:D\to A$ satisfies $f\circ k=g\circ k$
there exists a unique morphism $p:D\to C$ such that $h\circ p=k$.
\end{itemize}
\end{definition}
The situation is illustrated by the following commutative diagram:
\[
\xymatrix{
D \ar[dr]^k \ar@{-->}[d]_p^{!}
\\
C \ar[r]_h
&A \ar@<2pt>[r]^f \ar@<-2pt>[r]_g
&B
}
\]

\begin{remark}
As usual with universal properties, an equalizer, if it exists, is unique up to (unique) isomorphism.
In this case this means not only that for any other equalizer $k:L\to A$ the unique morphism $u:L\to C$
making the diagram
\[
\xymatrix{
L \ar[dr]^k \ar@{-->}[d]_u
\\
C \ar[r]^h
&A
}
\]
commute is an isomorphism,
but that conversely for any isomorphism $u:L\to C$
the morphism $h\circ u$ is an equalizer of $f,g$.
We will omit explicitly making similar remarks regarding other categorical gadgets.
\end{remark}

\begin{definition}
\emph{Coequalizer} is the dual of equalizer,
i.e., a \emph{coequalizer} of $f,g:A\to B$ is a morphism $h:B\to C$ such that
\begin{itemize}
\item
$h\circ f=h\circ g$;

\item
whenever $k:B\to D$ satisfies $k\circ f=k\circ g$
there exists a unique morphism $p:C\to D$ such that $k=p\circ h$.
\end{itemize}
\end{definition}
This is illustrated by the commutative diagram
\[
\xymatrix{
&&D
\\
A \ar@<2pt>[r]^f \ar@<-2pt>[r]_g
&B \ar[r]_h \ar[ur]^k
&C \ar@{-->}[u]_p^{!}
}
\]

\begin{remark}
If 
$f,g:A\to B$ and $q:B\to E$ is a monomorphism,
then $f,g$ and $q\circ f,q\circ g$ have the same equalizers.
Dually,
if $q:E\to A$ is an epimorphism,
then $f,g$ and $f\circ q,g\circ q$ have the same coequalizers.
\end{remark}

\begin{definition}
Let $\CC$  be a category with zero object 0, and let $f:A\to B$.
A \emph{kernel} of $f$ is an equalizer of the pair $f,0$
equivalently,
a
morphism $h:C\to A$ such that 
\begin{itemize}
\item  
$f\circ h = 0$;

\item  
whenever
$k: D\to A$  satisfies $f\circ k =  0$ there exists a unique morphism $p: D\to C$ such that $h\circ p = k$. 
\end{itemize}
\end{definition}

\begin{definition}
\emph{Cokernel} is the dual of kernel,
i.e., a \emph{cokernel} of $f:A\to B$
is a
coequalizer of $f,0$;
equivalently, a
morphism $h : B\to C$ such that 
\begin{itemize}
\item  
$h\circ f = 0$

\item  
whenever
$k: B\to D$  satisfies $k\circ f = 0$ there exists a unique morphism $p: C\to D$ such that $p\circ h = k$. 
\end{itemize}
\end{definition}

\begin{remark}
Every equalizer, and hence every kernel, is a monomorphism, and by duality every coequalizer, and hence every cokernel, is an epimorphism.
If 
$f:A\to B$ and $q:B\to E$ is a monomorphism, then
$f$ and $q\circ f$ have the same kernels.
Dually, if 
$q:E\to A$ is an epimorphism, then 
$f$ and $f\circ q$ have the same cokernels.
\end{remark}

\begin{remark}
If 0 is a zero object,
then for
all objects $A,B$
the morphism 
$0:A\to B$
has kernel $1_A$ and cokernel $0:B\to 0$.
\end{remark}

\begin{definition}
A \emph{split monomorphism} is a morphism $f:A\to B$
such that there exists a morphism $g:B\to A$ with $g\circ f=1_A$,
and dually a \emph{split epimorphism} is a morphism $f:A\to B$ such that there exists a morphism $g:B\to A$ with $f\circ g=1_B$.
\end{definition}

In the category-theory literature, one can find various definitions of image and coimage. A common definition of the image of a morphism uses subobjects, and we summarize this approach:
if $f:B\to A$ and $g:C\to A$ are two monomorphisms with common codomain $A$, write $f\le g$ to mean that $f=g\circ h$ for some $h$.
When both $f\le g$ and $g\le f$ write $f\sim g$.
This defines an equivalence relation among the monomorphisms
with codomain $A$,
and an equivalence class of these monomorphisms is called a \emph{subobject} of $A$.
The class (which could be proper) of all subobjects of $A$ is partially ordered by the binary relation ``$\le$''.

In practice, any monomorphism $f:B\to A$ is referred to as a subobject,
with the understanding that it is really just a representative of an equivalence class that is the actual subobject.

In some contexts the monomorphisms in the definition of subobject are required to satisfy some other conditions.
For example, one could restrict to \emph{strong monomorphisms} (see \remarkref{strong} for the definition).

In \cite[Section~I.10]{mitchellcategories} (for example) we find the following definition:

\begin{definition} 
The \emph{image} of a morphism $f:A\to B$ is the ``smallest'' subobject of $B$ through which $f$ factors, equivalently a monomorphism $u:I\to B$ such that  
\begin{itemize}
\item $f=u\circ f'$ for some $f':A\to I$
\item if $f=v\circ g$ for any other monomorphism $v:J\to B$ and a morphism $g:A\to J$, then there is a unique morphism $i:I\to J$ such that $u=v\circ i$ 
\end{itemize}
\end{definition}

In many categories this is a very useful definition, provided that it is not hard to determine what the subobjects are. For instance, in the category of sets subobjects are subsets, in the category of groups subobjects are subgroups, etc. However, as mentioned in the introduction, since we do not know what the monomorphisms --- or the strong monomorphisms, for that matter --- are in the enchilada category we were unable to determine what subobjects are. Therefore, we were unable to use this image definition. So, we use the following instead, which we took from \cite[Definition~12.3.7]{schubert}.

\begin{definition}\label{image def}
In a category with kernels and cokernels,
a \emph{Schubert image} of a morphism $f$
is a 
kernel of any cokernel of $f$,
and dually
a \emph{Schubert coimage} of $f$
is a cokernel of any kernel of $f$.
\end{definition}

We should note that Schubert defines an image of a morphism as above. However,  in a category where subobjects can be fully identified,
images (in the subobject sense) need not satisfy the condition of \defnref{image def}
unless (for example) the category is abelian. For instance, in the category of groups the above definition is applicable if and only if the image is a normal subgroup of the codomain. 

It is a trivial consequence of the definitions that the image of a morphism, if it exists, can be factored through the Schubert image.

It may be appropriate
to mention briefly the duals of subobjects and images.
If $f:A\to B$ and $g:A\to C$ are two epimorphisms with common domain $A$, write $f\le g$ to mean that $f=h\circ g$ for some $h$.
When both $f\le g$ and $g\le f$ write $f\sim g$.
This defines an equivalence relation among the epimorphisms
with domain $A$,
and an equivalence class of these epimorphisms is called a \emph{quotient object} of $A$. The \emph{coimage} of a morphism $f:A\to B$ is the smallest quotient object of $A$ through which $f$ factors.

\section{The enchilada category}\label{enchilada}

As we mentioned in the introduction,
in the enchilada category our objects are $C^*$-algebras, and the morphisms from $A$ to $B$ are the isomorphism classes of nondegenerate $A-B$ correspondences.
The paper \cite{taco} and the memoir \cite{enchilada} (particularly Chapter~1 and 2, and Appendix~A) contain a development of all the theory we will need,
but for those not familiar with $C^*$-correspondences we give a quick review:
a \emph{Hilbert $B$-module} is a vector space $X$ equipped with
a right $B$-module structure
and a \emph{$B$-valued inner product}, i.e., a positive-definite $B$-valued sesquilinear form $\<\cdot,\cdot\>_B$ satisfying
\[
\<x,yb\>_B=\<x,y\>_Bb\midtext{and}\<x,y\>^*_B=\<y,x\>_B
\]
for all $x,y\in X,b\in B$,
and which is complete in the norm $\|x\|=\|\<x,x\>_B\|^{1/2}$.
The closed span of the inner products is an ideal $B_X$ of $B$, and $X$ is called \emph{full} if $B_X=B$.
The $B$-module operators $T$ on $X$ for which there is an operator $T^*$ satisfying
\[
\<Tx,y\>_B=\<x,T^*y\>_B\righttext{for all}x,y\in X
\]
(which is not automatic, even if $T$ is bounded)
form the $C^*$-algebra $\LL(X)$ of \emph{adjointable operators} with the operator norm,
and the closed linear span of the \emph{rank-one operators} $\theta_{x,y}$ given by
\[
\theta_{x,y}z=x\<y,z\>_B
\]
is the closed ideal $\KK(X)$ of \emph{compact operators}.

By an \emph{$A-B$ correspondence} $X$
we mean a Hilbert $B$-module $X$ with a *-homomorphism\footnote{and henceforth we will drop the *, so that all homomorphisms are assumed to be *-homomorphisms} $\phi_X:A\to \LL(X)$,
and we say the correspondence is \emph{nondegenerate} if $AX=X$.\footnote{Note that we actually mean $AX=\{ax:a\in A,x\in X\}$ --- by the Cohen-Hewitt factorization theorem this coincides with the closed span.}
\emph{All our correspondences will be nondegenerate by standing hypothesis.}
That is, from now on when we use the term correspondence we will tacitly assume the nondegeneracy condition.
An \emph{isomorphism} $U:X\to Y$ of $A-B$ correspondences is a linear bijection
such that
\[
U(axb)=aU(x)b\midtext{and}\<Ux,Uy\>_B=\<x,y\>_B
\]
for all $a\in A,b\in B,x,y\in X$.

The \emph{balanced tensor product} $X\otimes_BY$ of
an $A-B$ correspondence $X$ and a $B-C$ correspondence $Y$ is
formed as follows:
the algebraic tensor product $X\odot Y$
is given the $A-C$ bimodule structure determined on the elementary tensors by
\[
a(x\otimes y)c=ax\otimes yc
\righttext{for}a\in A,x\in X,y\in Y,c\in C,
\]
and the unique $C$-valued sesquilinear form whose values on elementary tensors are given by
\[
\<x\otimes y,u\otimes v\>_C=\bigl\<y,\<x,u\>_Bv\>_C
\righttext{for}x,u\in X,y,v\in Y.
\]
The Hausdorff completion is an $A-C$ correspondence $X\otimes_B Y$.
The term \emph{balanced} refers to the property
\[
xb\otimes y=x\otimes by
\righttext{for}x\in X,b\in B,y\in Y,
\]
which is automatically satisfied.
The \emph{identity correspondence} on $A$ is the vector space $A$ with the $A-A$ bimodule structure given by multiplication and the inner product
$\<a,b\>_A=a^*b$.
The \emph{enchilada category} has
$C^*$-algebras as objects,
and the morphisms from $A$ to $B$ are the isomorphism classes of $A-B$ correspondences,
with composition given by balanced tensor product and identity morphisms given by identity correspondences.

We write ``$\pre AX_B$ is a correspondence'' to mean that $X$ is an $A-B$ correspondence,
and we write $[X]=[\pre AX_B]$ for the associated morphism in the enchilada category.
Unless otherwise specified, $\pre AA_A$ will mean the identity correspondence over $A$.
Note that composition is given by
\[
[\pre BY_C]\circ [\pre AX_B]=[\pre A(X\otimes_B Y)_C].
\]
Actually, we will frequently drop the square brackets $[\cdot]$, since it will clean up the notation and no confusion will arise.

The \emph{multiplier algebra} of $A$ is the $C^*$-algebra $M(A)=\LL(\pre AA_A)$,
and we identify $A$ with its image under the left-module homomorphism $\phi_A:A\to M(A)$.
In this way $A$ becomes an ideal of $M(A)$.
More generally, it is a standard fact that $M(\KK(X))=\LL(X)$.
A homomorphism $\mu:A\to M(B)$ is \emph{nondegenerate} if $\mu(A)B=B$,
and nondegeneracy of a correspondence $\pre AX_B$ is equivalent to nondegeneracy of the left-module homomorphism $\phi_X:A\to M(\KK(X))$.
Every nondegenerate homomorphism $\mu:A\to M(B)$ extends uniquely to a homomorphism $\bar\mu:M(A)\to M(B)$,
and we typically drop the bar, just writing $\mu$ for the extension.

An \emph{$A-B$ \hbm} is an $A-B$ correspondence $X$
that is also equipped with an $A$-valued inner product $\pre A\<\cdot,\cdot\>$,
which satisfies the ``mirror image'' of the properties of the $B$-valued inner product:
\[
\pre A\<ax,y\>=a\pre A\<x,y\>\midtext{and}\pre A\<x,y\>^*=\pre A\<y,x\>
\]
for all $a\in A,x,y\in X$,
as well as the compatibility property
\[
\pre A\<x,y\>z=x\<y,z\>_B
\righttext{for}x,y,z\in X.
\]
A \hbm\ $\pre AX_B$ is \emph{left-full} if the closed span $A_X$ of $\pre A\<X,X\>$ is all of $A$
(and to avoid confusion we sometimes refer to the property $\clspn\<X,X\>_B=B$ as \emph{right-full}).
In any event, if $X$ is an $A-B$ \hbm\ then
$A_X$ is an ideal of $A$ that is mapped isomorphically onto $\KK(X)$ via $\phi_X$.
The \emph{dual} $B-A$ \hbm\ $\wilde X$ is the formed as follows:
write $\wilde x$ when a vector $x\in X$ is regarded as belonging to $\wilde X$,
define the $B-A$ bimodule structure by
\[
b\wilde x a=\wilde{a^*xb^*}
\]
and the inner products by
\[
\pre B\<\wilde x,\wilde y\>=\<x,y\>_B\midtext{and}\<\wilde x,\wilde y\>_A=\pre A\<x,y\>
\]
for $b\in B,x,y\in X,a\in A$.
An \emph{$A-B$ \ibm} is an $A-B$ \hbm\ that is full on both the left and the right.
Note that every Hilbert $B$-module may be regarded as a left-full $\KK(X)-B$ \hbm\ with left inner product
\[
\pre{\KK(X)}\<x,y\>=\theta_{x,y}.
\]
It is a fundamental fact about the enchilada category that the invertible morphisms are precisely the (isomorphism classes of) \ibm s (see, for example, \cite[Proposition~2.6]{taco}, \cite[Lemma~2.4]{enchilada}, and \cite[Proposition~2.3]{schweizer}).

In  this section, we show the existence of the necessary ingredients, such as kernel and image, to construct an exact sequence in the enchilada category.
As we mentioned in the introduction, we must be careful in defining the image of a morphism.

It 
is obvious
that in the enchilada category
a 0 object is any 0-dimensional $C^*$-algebra, and
the 0 morphism from $A$ to $B$ is the 0 correspondence $\leftidx{_A}{0}_B$.
If $\mu:A\to M(B)$ is a homomorphism, then $\mu(A)B$ is an $A-B$ correspondence.
When $A$ is a closed ideal\footnote{and henceforth we will drop ``closed'', so that all ideals are tacitly assumed to be closed} of $B$, we get a correspondence $\pre AA_B$.
For any ideal $I$ of a $C^*$-algebra $B$,
the quotient map $B\to B/I$ gives rise to a correspondence 
$\leftidx{_B}(B/I)_{B/I}$.

\begin{prop}\label{0 tensor}
Given correspondences
$\leftidx{_A}X_B$ and $\leftidx{_B}Y_C$,
we have
$X{\otimes}_B Y =  0$ if and only if 
\[
B_X\subset \ker\phi_Y.
\]
\end{prop}

\begin{proof}
Assume that $X{\otimes}_B Y = 0$. Let $x_1,  x_2 \in X$ and $y_1, y_2 \in Y$. Then we have, 
\[
0=\la x_1{\otimes}_B y_1, x_2{\otimes}_B y_2\ra_C = \la y_1, \la x_1, x_2\ra_B y_2\ra_C.
\]
This implies that $\la x_1, x_2\ra_B y =0$ for all $y \in Y$,
i.e.,
$\la x_1, x_2\ra_B\in \ker\phi_Y$.  Since any element of $B_X$ is a limit of linear combinations of elements $\la x_1, x_2\ra_B$ where $x_i \in X$, we conclude that $B_X \subset \ker\phi_Y$. 

Now assume that $B_X\subset \ker\phi_Y$.
Of course, since
$X\otimes_B Y$ is the closed span of elementary tensors,
in order to show $X\otimes_B Y=0$ it suffices to show that $x\otimes y=0$ for any $x\in X$ and $y\in Y$:
by the Cohen-Hewitt factorization theorem we can write x as $x_1 b$ for some $x_1\in X$ and $b\in B_X$, and then we have
\[
x\otimes_B y=x_1 b {\otimes}_B y= x_1{\otimes}_B b y=x_1{\otimes}_B \phi_Y(b)y=0.
\]
\end{proof}

\begin{lemma}\label{restrict right}
Let $\pre AX_B$ be a correspondence and $C$ be a $C^*$-subalgebra of $B$ containing $B_X$. 
Then $X$ becomes an $A-C$ correspondence $\pre AX_C$ by restricting the right-module structure to $C$,
and the map
\[
x\otimes b\mapsto xb\righttext{for}x\in X,b\in CB
\]
extends uniquely to an isomorphism
\[
\pre AX_C\otimes_C(CB)_B\cong \pre AX_B.
\]

\end{lemma}
Note that $CB$ is a closed right ideal of $B$, by the Cohen-Hewitt factorization theorem.

\begin{proof}
$X$ already has a right $B$-module structure. Restricting this to $C$, we get a right $C$-module structure. Now, we need an inner product into $C$, which we get directly since  $B_X\subset C$. 
Then
the standard computation
\begin{align*}
\<x\otimes b,x'\otimes b'\>_B
&=b^*\<x,x'\>_Bb'
=\<xb,x'b'\>_B
\end{align*}
implies the assertion regarding the isomorphism.
\end{proof}
In Lemma~\ref{restrict right}, $C$ will usually be an ideal of $B$, and then $\pre AX_C\otimes_C \pre CC_B\cong \pre AX_B$.
A frequently used special case is when $C=B$,
and then the main content is the isomorphism
$X\otimes_BB\cong X$.

\begin{lemma}\label{compose left}
Let $X$ be a Hilbert $B$-module,
and let
$\pi:A\to C$ be a nondegenerate homomorphism.
\begin{enumerate}
\item
Let $\pre AX_B$ be a correspondence,
and
suppose that $\ker\pi\subset \ker\phi_{A,X}$.
Let $\phi_{C,X}$ be the unique homomorphism
making the diagram
\[
\xymatrix{
A \ar[r]^-{\phi_{A,X}} \ar[d]_\pi
&\LL(X)
\\
C \ar[ur]_{\phi_{C,X}}
}
\]
commute,
and let $\pre CX_B$ be the associated correspondence.
Then the map
\[
c\otimes x\mapsto \phi_{C,X}(c)x\righttext{for}c\in C,x\in X
\]
extends uniquely to an isomorphism
\[
\pre AC_C\otimes_C \pre CX_B\cong \pre AX_B
\]
of $A-B$ correspondences.

\item
Suppose that $\pi$ is surjective.
Let $\pre CX_B$ and $\pre CY_B$ be correspondences,
and define correspondences $\pre AX_B$ and $\pre AY_B$ by
\[
\phi_{A,X}=\phi_{C,X}\circ\pi
\midtext{and}
\phi_{A,Y}=\phi_{C,Y}\circ\pi.
\]
Then
$\pre AX_B\cong \pre AY_B$
if and only if $\pre CX_B\cong \pre CY_B$.
\end{enumerate}
\end{lemma}

\begin{proof}
(1) is folklore.
For (2),
first note that one direction follows quickly from part~(1):
if $\pre CX_B\cong \pre CY_B$
then
\[
\pre AX_B\cong C\otimes_C \pre CX_B\cong C\otimes_C \pre CY_B\cong \pre AY_B.
\]
Conversely
let
\[
U:\pre AX_B\variso \pre AY_B
\]
be an isomorphism.
Then $\ad U:\LL_B(X)\to \LL_B(Y)$ is also an isomorphism.
For $c\in C$ we can choose $a\in A$ such that $\pi(a)=c$,
and then
\begin{align*}
\ad U\circ\phi_{C,X}(c)
&=\ad U\circ\phi_{C,X}\circ\pi(a)
\\&=\ad U\circ\phi_{A,X}(a)
\\&=\phi_{A,Y}(a)
\\&=\phi_{C,Y}\circ\pi(a)
\\&=\phi_{C,Y}(c),
\end{align*}
so that $U$ also preserves the left $C$-module structures.
\end{proof}
Frequently-used
special cases of Lemma~\ref{compose left},
for a given $A-B$ correspondence $X$,
are
the isomorphisms
\begin{itemize}
\item
$A\otimes_AX\cong X$
(where $\pi=\id_A$),

\item
$A/I\otimes_{A/I}X'\cong X$
(when 
$I$ is an ideal of $A$ contained in $\ker\phi_X$),
and

\item
$\KK(X)\otimes_{\KK(X)}X'\cong X$
\cite[discussion preceding Proposition~2.27]{enchilada}.
\end{itemize}

In connection with 
item (2) of Lemma~\ref{compose left},
there is more to say:
\begin{prop}\label{epi}
If $\pi:A\to C$ is a surjective homomorphism,
then $\pre AC_C$ is an epimorphism in the enchilada category.
\end{prop}

\begin{proof}
Given $C-B$ correspondences $X$ and $Y$
such that
\[
\pre AC\otimes_C X_B\cong \pre AC\otimes_C Y_B,
\]
we must show that
$\pre CX_B\cong \pre CY_B$.
Using Lemma~\ref{compose left},
we can regard $X$ and $Y$ as $A-B$ correspondences $\pre AX_B$ and $\pre AY_B$,
and then part (1) of the lemma and the hypothesis together tell us that 
$\pre AX_B\cong \pre AY_B$,
and so by part (2) of the lemma we also have
$\pre CX_B\cong \pre CY_B$.
\end{proof}

\begin{prop}\label{mono inj}
If $\pre AX_B$ is a monomorphism in the enchilada category, then $\phi_X: A\to \LL_{B}(X)$ is injective. 
\end{prop}

\begin{proof}
Assume that $\phi_X$ is not injective. Then, $K=\ker\phi_X$ is a non-zero ideal of $A$. Consider the correspondence ${\leftidx_{K}K}_A$.  Since $\<K,K\>_K=K$, by Proposition~\ref{0 tensor} we have
\[
K\otimes_A X=0=0\otimes_A X,
\]
but $K\neq 0$, so $X$ is not a monomorphism. 
\end{proof}

We  show that the converse of \propref{mono inj}  is not true in general:

\begin{prop}
There exists an injective \C-correspondence that is not a monomorphism in the enchilada category.
\end{prop}

\begin{proof}

Let  $H$ be an infinite dimensional Hilbert space. Then we may view $H$ as an injective \C-correspondence over $\CC$.
Consider the usual Hilbert spaces $\CC$ and $\CC^2$.
We have the isomorphism 
\[
\CC\otsc H\cong \CC^2\otsc H,
\]
but $\CC$ and $\CC^2$ are not isomorphic as $\CC-\CC$ correspondences.
\end{proof}

\begin{lemma}\label{X tilde X}
If $X$ is an $A-B$ Hilbert bimodule, then
\begin{align*}
\wilde X\otimes_AX&\cong B_X\righttext{as $B-B$ correspondences}
\\
X\otimes_B\wilde X&\cong A_Y\righttext{as $A-A$ correspondences.}
\end{align*}
\end{lemma}

\begin{proof}
This is folklore.
If $X$ is an \ibm, then this is \cite[Proposition~3.28]{tfb} (for example),
and the general case can be proved using the same techniques.
\end{proof}

\begin{lemma}\label{BBX}
If $\pre AX_B$ and $\pre BY_C$ are correspondences, then
\[
X\otimes_BY\cong X\otimes_{B_X}Y
\]
as $A-C$ correspondences.
\end{lemma}

\begin{proof}
Just note that the
balancing relations determined by $\otimes_B$ and $\otimes_{B_X}$ coincide:
for every $x\in X$, $b\in B$, and $y\in Y$
we can choose $x'\in X,b'\in B_X$ such that $x=x'b'$,
and then
the following computation in $X\otimes_{B_X} Y$ suffices:
\begin{align*}
xb\otimes y
&=x'b'b\otimes y
\\&=x'\otimes b'by\righttext{(since $b'b\in B_X$)}
\\&=x'b'\otimes by
\\&=x\otimes by.
\end{align*}
\end{proof}

\begin{prop}\label{split mono is left-full}
A  left-full \hbm is a split monomorphism in the enchilada category.
\end{prop}

\begin{proof}
Assume that $X$ is a left-full \hbm.
Then we have a $B-A$
Hilbert bimodule $\wilde X$, and
\begin{align*}
[\wilde X]\circ [X]
&=[X\otimes_B\wilde X]
\\&=[X\otimes_{B_X} \wilde X]\righttext{(by Lemma~\ref{BBX})}
\\&=[A],
\end{align*}
so $[X]$ is a split monomorphism
\end{proof}

We now show that inverse of Proposition~\ref{split mono is left-full} is not true in general.

\begin{prop}\label{hb}
The Enchilada category has a split monomorphism that is not the isomorphism class of a Hilbert bimodule. However, if the isomorphism class of a  Hilbert bimodule \AXB\ is a split monomorphism, then \AXB\ has to be left full. 
\end{prop}

\begin{proof} Let $X$ be the injective $\CC - \CC^2$ correspondence associated to the homomorphism $a\mapsto (a,a),$ and let $Y$ be the $\CC^2-\CC$ correspondence associated the homomorphism $(a,b)\mapsto a.$ Then, we have 
\[ \pre {\CC}(X{\otimes}_{\CC^2} Y)_{\CC} \cong \pre {\CC}\CC_{\CC}. \]
However, the correspondence $\pre \CC(X)_{\CC^2}$  is not a Hilbert bimodule.

To prove the second half of the proposition, let  \AXB\ be a Hilbert bimodule such that [\AXB] is a split monomorphism in the enchilada category. Then, there exists a correspondence $\pre BY_A$  such that 
\[ \pre AX_B \otss \pre BY_A \cong \pre AA_A.\]
Then the $C^*$-correspondence $\pre A(X\otss Y)_A$ must be right full, i.e, 
\[ \<Y, \<X,X\>_B\cdot Y\>_A = A. \]
Let $Z=B_XY$.
Note that
\[
X\otimes_B Y
=X\otimes_{B_X} Y
=X\otimes_{B_X} Z
\]
as $A-A$ correspondences.
Thus
\begin{align*}
\X\otimes_A X\otimes_B Y
&=\X\otimes_A X\otimes_{B_X} Y
\\&\cong B_X\otimes_{B_X} Y
\\&\cong B_X Y
\\&=Z
\end{align*}
as $B-A$ correspondences.

Thus
\begin{align*}
A_X
&\cong A_X\otimes_A A
\\&\cong X\otimes_B \X\otimes_A X\otimes_B Y
\\&\cong X\otimes_B Z
\\&=X\otimes_B Y
\\&\cong A
\end{align*}
as $A-A$ correspondences. Therefore the ideal $A_X$ must be all of $A$.
\end{proof}

\begin{theorem}\label{ker X}
Let X be an $A-B$ correspondence and let $K$ be the kernel of the associated homomorphism $\phi_X:A\to \LL_B$$(X)$. Then
the correspondence $\pre KK_A$ is a kernel of $X$.
\end{theorem}

\begin{proof}
First, $K\otimes_KX=0$ by Lemma~\ref{0 tensor}, because $\<K,K\>=K=\ker \phi_X$.

Now suppose that $\pre CY_A$ is a correspondence such that $Y\otimes_AX=0$.
Then by Lemma~\ref{0 tensor} we have
\[
A_Y\subset \ker\phi_X=K,
\]
so by Lemma~\ref{restrict right} we 
get
a 
correspondence $\pre CY_K$
such that
\[
\pre CY_K\otimes_KK_A\cong \pre CY_A.
\]
Moreover, $\pre CY_K$ is unique up to isomorphism because $\pre KK_A$ is a
left-full \hbm,
and hence is a split monomorphism by Proposition~\ref{split mono is left-full},
so in particular is a monomorphism.
\end{proof}

\begin{corollary}\label{left-full is kernel}
A correspondence $\pre AX_B$ is a left-full Hilbert bimodule
if and only if
$X$ is a kernel in the enchilada category,
in which case it is a kernel of $\pre B(B/B_X)_{B/B_X}$.
\end{corollary}

\begin{proof}
First assume that $\pre AX_B$ is a kernel of a correspondence $\pre BY_C$.
Then by Theorem~\ref{ker X} $\pre AX_B$ is isomorphic to the kernel $\pre KK_B$,
where $K=\ker \phi_Y$,
in the sense that there is 
an \ibm\ $Y$ 
(equivalently, $[Y]:A\to K$ is an isomorphism in the enchilada category)
making the diagram
\[
\xymatrix{
A \ar[r]^-X \ar[d]_Y
&B
\\
K \ar[ur]_K
}
\]
commute.
Since $\pre KK_B$ is a left-full \hbm, so is $\pre AX_B$.

Conversely, assume that $X$ is a left-full \hbm.
By Lem\-ma~\ref{restrict right} we can regard $X$ as an 
\ibm\ $\pre AX_{B_X}$, and
\[
\pre AX_B\cong \pre AX_{B_X}\otimes_{B_X} \pre {B_X}(B_X)_B.
\]
By Theorem~\ref{ker X}, 
$\pre {B_X}(B_X)_B$
is a kernel of 
$\pre B(B/B_X)_{B/B_X}$,
and hence so is $\pre AX_B$
because $\pre AX_{B_X}$ is an \ibm.
\end{proof}

\begin{prop}\label{coker}
A
correspondence $\pre AX_B$ has cokernel $\pre B(B/B_X)_{B/B_X}$.
\end{prop}

\begin{proof}
First, $X\otimes_B B/B_X=0$,
because $B_X$ is the kernel of the quotient map $B\to B/B_X$.
Now suppose that $\pre BY_C$ is a correspondence such that
$X\otimes_BY=0$.
Then $B_X\subset \ker\phi_Y$ by Lemma~\ref{0 tensor},
so by Lemma~\ref{compose left} we may regard $Y$ as a $B/B_X-C$ correspondence,
and we have
\begin{equation}\label{factor}
\pre B(B/B_X)_{B/B_X}\otimes_{B/B_X} \pre {B/B_X}Y_C\cong \pre BY_C.
\end{equation}
Moreover,
by Proposition~\ref{epi} $\pre B(B/B_X)_{B/B_X}$
is an epimorphism in the enchilada category,
so \eqref{factor} determines $\pre {B/B_X}Y_C$ up to isomorphism.
\end{proof}

Finally, we are ready for images. 

\begin{theorem}\label{image thm}
A
correspondence $\pre AX_B$ has Schubert image $\pre {B_X}(B_X)_B$ and Schubert coimage $\pre A(A/\ker\phi)_{A/\ker\phi}$,
where $\phi=\phi_X$ is the associated homomorphism.
\end{theorem}

\begin{proof}
By Proposition~\ref{coker} $\pre B(B/B_X)_{B/B_X}$ is a cokernel of $X$,
and by 
\theoremref{ker X}
$\pre {B_X}(B_X)_B$ is a kernel of $\pre B(B/B_X)_{B/B_X}$.
Thus $\pre {B_X}(B_X)_B$ is a cokernel of a kernel of $X$,
and so is the Schubert image of $X$ by definition.

Very similarly, $\pre {\ker\phi}(\ker\phi)_A$ is a kernel of $X$, and $\pre A(A/\ker\phi)_{A/\ker\phi}$ is a cokernel of $\pre {\ker\phi}(\ker\phi)_A$. Thus, by definition, it is the Schubert coimage of $X$. 
\end{proof}

\begin{prop}
Every Hilbert bimodule $\pre AX_B$ has an image, and in fact it coincides with the Schubert image.
\end{prop}

\begin{proof}
Assume that $\pre AX_B$ is isomorphic to $\pre AY{\otimes}_C Z_B$ for a mono\-morph\-ism
$\pre CZ_B$ and a morphism $\pre AY_C$. Since $\pre AX_{B_X}$ is a Hilbert bimodule, there exists a ${B_X}-A$ Hilbert bimodule   $\wilde X$ such that $\wilde X{\otimes}_A X \cong B_X$ as ${B_X}-{B_X}$ correspondences (Lemma~\ref{X tilde X}). Denote $\wilde X{\otimes}_A Y$ by $M$. Then we have 

\[
\pre{B_X}M{\otimes}_{C} Z_B\cong \pre{B_X}\wilde X{\otimes}_A Y{\otimes}_{C} Z_B \cong \pre{B_X}\wilde X{\otimes}_A X_B\cong \pre{B_X}(B_X)_B.
\]
Thus $\pre{B_X}M_B$ is the unique monomorphism making the diagram 
\[
\xymatrix@R+20pt{
A \ar[rr]^-X \ar[dr]^{X'} \ar[ddr]_Y
&&B 
\\
&B_X \ar[ur]^{B_X} \ar@{-->}[d]|(.4)M
\\
&C \ar[uur]_Z
}
\]
commute, which completes the proof.
\end{proof} 

And \theoremref{image thm} in turn makes us ready for exact sequences:

\begin{definition}
A sequence
\[
\xymatrix{
\cdots
\ar[r]
&A \ar[r]^-{[X]}
&B \ar[r]^-{[Y]}
&C \ar[r]
&\cdots
}
\]
of morphisms
in the enchilada category is \emph{exact at $B$} if
the image of $X$ equals the kernel of $Y$,
and is \emph{exact} if it is exact at every node.
\end{definition}

\begin{theorem}\label{short exact}
Let $\pre AX_B$ and $\pre BY_C$ be correspondences.
Then the sequence
\[
\xymatrix{
0\ar[r] &A \ar[r]^-X &B \ar[r]^-Y &C \ar[r] &0
}
\]
is exact in the enchilada category if and only if
$\phi_X$ is injective,
$B_X=\ker\phi_Y$,
and 
$Y$ is full
\(i.e., $C_Y=C$\).
\end{theorem}

\begin{proof}
Since $A_0=\<0,0\>=0$, 
the morphism
$0:0\to A$ is an image of $0:0\to A$.
On the other hand,
$\ker\phi_X:\ker\phi_X\to A$ is a kernel of $X:A\to B$.
Thus the sequence is exact at $A$ if and only if $\phi_X$ is injective.

Next,
$C_Y:C_Y\to C$ is an image of $Y:B\to C$.
On the other hand, the homomorphism $\phi_0:C\to \LL(0)$ is 0, so has kernel $C$.
Thus $C:C\to C$ is a kernel of $0:C\to 0$.
Therefore the sequence is exact at $C$ if and only if
$C_Y=C$.

Finally,
$B_X:B_X\to C$ is an image of $X:A\to B$,
and
$\ker\phi_Y:\ker\phi_Y\to B$ is a kernel of $Y:B\to C$,
so the sequence is exact at $B$
if and only if $B_X=\ker\phi_Y$.
\end{proof}

\begin{remark}\label{failure}
As we mentioned in the introduction,
our primary motivation for investigating exact sequences in the enchilada category
was to determine whether the reduced-crossed-product functor is exact in the enchilada categories,
which would obviously be relevant for the Baum-Connes conjecture.
But now we will show that it is not exact.
Let
\[
\xymatrix{
0 \ar[r] &A \ar[r]^-X &B \ar[r]^-Y &C \ar[r] &0
}
\]
be a short exact sequence of correspondences,
and let $G$ be a locally compact group.
Further let
$\alpha:G\to \aut A$,
$\beta:G\to \aut B$,
$\gamma:G\to \aut C$,
$\zeta:G\to \aut X$,
and
$\eta:G\to \aut Y$
be actions
as in \cite[Section~3.1.1]{enchilada}.
Then we can form the reduced-crossed-product correspondences
\[
\pre {A\rtimes_{\alpha,r} G}(X\rtimes_{\zeta,r} G)_{B\rtimes_{\beta,r} G}
\midtext{and}
\pre {B\rtimes_{\beta,r} G}(Y\rtimes_{\eta,r} G)_{C\rtimes_{\gamma,r} G}.
\]
Moreover,
there are actions
\[
\mu:G\to \aut \KK(X)
\midtext{and}
\nu:G\to \aut \KK(Y)
\]
such that
\begin{align*}
\KK(X\rtimes_{\zeta,r} G)&=\KK(X)\rtimes_{\mu,r} G
\\
\phi_{X\rtimes_{\zeta,r} G}&=\phi_X\rtimes_r G
\\
\KK(Y\rtimes_{\eta,r} G)&=\KK(Y)\rtimes_{\nu,r} G
\\
\phi_{Y\rtimes_{\eta,r} G}&=\phi_Y\rtimes_r G.
\end{align*}
Now consider the following sequence
\[
\xymatrix@C+10pt{
0 \ar[r]
&A\rtimes_{\alpha,r} G \ar[r]^-{X\rtimes_{\zeta,r} G}
&B\rtimes_{\beta,r} G \ar[r]^-{Y\rtimes_{\eta,r} G}
&C\rtimes_{\gamma,r} G \ar[r]
&0
}
\]
One of the required conditions to make this sequence exact is
\[
(B{\rtimes}_{\beta,r} G )_{X{\rtimes}_{\zeta,r} G} = \ker \phi_{Y\rtimes_{\eta,r} G}.
\]
Now, by \cite[Proposition~3.2]{enchilada} we have $(B{\rtimes}_{\beta,r} G )_{X{\rtimes}_{\zeta,r} G}=B_X{\rtimes}_{\beta,r} G$.
Since $\phi_{Y\rtimes_{\eta,r} G}=\phi_Y\rtimes_r G$
and $B_X=\ker\phi_Y$,
for exactness we must have
\[
(\ker\phi_X)\rtimes_{\alpha,r} G
=
\ker(\phi_X\rtimes_r G).
\]
However
this equality
is not true 
in general if $G$ is not exact.
So, the reduced-crossed-product functor does not preserve exactness in the enchilada categories if $G$ is a nonexact group.
\end{remark}

\begin{remark}
Although it is not directly relevant for our original investigation regarding the Baum-Connes conjecture,
we will now point out that
the full-crossed-product functor is 
also not
exact in the enchilada categories.

Again let
\[
\xymatrix{
0\ar[r]&A\ar[r]^-X&B\ar[r]^-Y&C\ar[r]&0
}
\]
be a short exact sequence of correspondences,
carrying compatible actions of $G$.
For exactness, we would need the sequence
\begin{equation}\label{crossed product sequence}
\xymatrix@C+10pt{
0\ar[r]
&A\rtimes G\ar[r]^-{X\rtimes G}
&B\rtimes G\ar[r]^-{Y\rtimes G}
&C\rtimes G\ar[r]&0
}
\end{equation}
to be exact.
To apply
Theorem~\ref{short exact},
we would need:
\begin{itemize}
\item
$\phi_{X\rtimes G}$ to be injective,

\item
$(B\rtimes G)_{X\rtimes G}=\ker\phi_{Y\rtimes G}$,
and

\item
$(C\rtimes G)_{Y\rtimes G}=C\rtimes G$.
\end{itemize}
This time, it
may fail to be exact at $A\rtimes G$,
i.e., $\phi_{X\rtimes G}$ need not be injective.
To explain all this,
note that, as for reduced crossed products,
\begin{align*}
\KK(X\rtimes G)&=\KK(X)\rtimes G
\\
\KK(Y\rtimes G)&=\KK(Y)\rtimes G
\\
\phi_{X\rtimes G}&=\phi_X\rtimes G
\\
\phi_{Y\rtimes G}&=\phi_Y\rtimes G.
\end{align*}
Then we have
\begin{align*}
(B\rtimes G)_{X\rtimes G}
&=B_X\rtimes G
\\&=(\ker\phi_Y)\rtimes G\righttext{(since $B_X=\ker\phi_Y$)}
\\&=\ker(\phi_Y\rtimes G)
\\&=\ker \phi_{Y\rtimes G},
\end{align*}
so \eqref{crossed product sequence} is exact in the middle.

Next,
\begin{align*}
(C\rtimes G)_{Y\rtimes G}
&=C_Y\rtimes G\righttext{(since $C_Y=C$)}
\\&=C\rtimes G,
\end{align*}
so \eqref{crossed product sequence} is exact at $C\rtimes G$.

To see how exactness at $A\rtimes G$ might fail,
we formulate a strategy for finding a counterexample:
we will take a correspondence $X$ 
arising from 
a homomorphism, 
and 
we will let
all the actions 
be trivial.
More precisely,
we will have:
\begin{itemize}
\item
an injection $\pi$ of $A$ into an ideal $D$ of $B$,
and

\item
the associated $A-B$ correspondence $X=AB$.
\end{itemize}
Since all the actions are trivial, the crossed products are just the maximal tensor products with $C^*(G)$.
In particular,
the correspondence
$\pre{A\rtimes G}(X\rtimes G)_{B\rtimes G}$
is the tensor product
\[
\pre{A\xm C^*(G)}(X\xm C^*(G))_{B\xm C^*(G)}.
\]
It follows from the standard theory of $C^*$-correspondences
that this is isomorphic to
\[
\pre AX_B\xm C^*(G).
\]
In particular, the left-module homomorphism
$\phi_{X\rtimes G}$ becomes
the tensor product
\[
\pi\xm\id:A\xm C^*(G)\to B\xm C^*(G).
\]
If we take any group $G$ for which $C^*(G)$ is nonnuclear
(for example, the free group on 2 generators),
then it is a subtle fact from $C^*$-algebra theory
\cite[Theorem~IV.3.1.12]{blackadar}
that there exist a $C^*$-algebra
$D$ and a $C^*$-subalgebra $A$
such that
the associated homomorphism
$\pi\xm\id$
is noninjective.
Then for any full correspondence $\pre BY_C$
with $\ker\phi_Y=D$
we get a counterexample.
\end{remark}

\section{Further properties of the enchilada category}\label{further}

In Corollary~\ref{left-full is kernel} we characterized kernels in the enchilada category as the left-full \hbm s.
In view of other results that come in dual pairs, it is tempting to suspect that cokernels are precisely the right-full \hbm s.
On the other hand, the enchilada category is decidedly left-challenged, so it is not surprising that it could have direction-related properties
that are not satisfied when the
directions are reversed.
Indeed:
\begin{prop}\label{right-full is not cokernel}
In the enchilada category, a cokernel need not be a right-full \hbm,
and conversely a right-full \hbm\ need not be a cokernel.
\end{prop}

\begin{proof}
First, it follows from Proposition~\ref{coker} that a quotient map $B\to B/J$ is a cokernel in the enchilada category.
However, it need not be a \hbm, since the quotient $B/J$ need not be isomorphic to an ideal of $B$.

Conversely, let $A$ be a nonzero $C^*$-algebra.
Then $\pre AA_A$ is a right-full \hbm,
and we will show that it is not a cokernel.
Arguing by contradiction, suppose $\pre AA_A$ is a cokernel of $\pre BX_A$.
Then
\[
X\otimes_AA=0,
\]
so $A_X\subset \ker \phi_A=0$,
and hence $X=0$.
Now, 
$0:A\to 0$ is
also a cokernel of $0:B\to A$,
so $A$ and 0 are isomorphic in the enchilada category, and hence are Morita equivalent.
Therefore $A=0$,
which is a contradiction.
\end{proof}

\begin{prop}\label{epi is full}
If a correspondence $\pre AX_B$ is an epimorphism in the enchilada category, then $X$ is full, i.e., $B_X=B$.
\end{prop}

\begin{proof}
Suppose that $B_X\ne B$.
Then
\[
X\otimes_B \pre B(B/B_X)_{B/B_X}=0=X\otimes_B 0,
\]
but $\pre B(B/B_X)_{B/B_X}\not\cong 0$.
\end{proof}

\propref{epi is full} can be alternatively restated as follows:
$\pre AX_B$ if is an epimorphism then its image is the identity morphism $\pre BB_B$.
In many categories,
there is a converse:
a morphism $f:A\to B$ whose image is the identity morphism $1_B$ must be an epimorphism --- informally, surjections are epimorphisms.
But not in the enchilada category:

\begin{prop}\label{sur not epi}
The enchilada category has a morphism $X:A\to B$
that is not an epimorphism
but whose image is $\pre BB_B$.
\end{prop}

\begin{proof}
We must find correspondences $\pre AX_B$, 
$\pre BY_C$, and $\pre BZ_C$
such that
\begin{enumerate}
\item
$B_X=B$,

\item
$X\otimes_BY\cong X\otimes_BZ$,
and

\item
$Y\not\cong Z$.
\end{enumerate}
We let
$X$ be the $\CC-\CC^2$ correspondence
associated to the homomorphism $a\mapsto (a,a)$,
and
for $Y$ and $Z$ we take the $\CC^2-\CC$ correspondences
associated to the homomorphisms
$(a,b)\mapsto a$
and
$(a,b)\mapsto b$, respectively.

It follows from Lemma~\ref{compose left}
that
\[
X\otimes_BY\cong \pre \CC\CC_\CC\cong X\otimes_BZ.
\]
On the other hand,
the $\CC^2-\CC$ correspondences $Y$ and $Z$ are
not isomorphic,
since
by
\cite[Proposition~2.3]{taco} two $B-C$ correspondences coming from nondegenerate homomorphisms $\pi,\rho:B\to M(C)$
are isomorphic if and only if there is a unitary $u\in M(C)$ such that $\ad u\circ\pi=\rho$.
\end{proof}

In spite of \propref{sur not epi}, there is a weaker result,
namely \propref{epi}.

\begin{corollary}\label{no equalizers}
The enchilada category does not have equalizers.
\end{corollary}

\begin{proof}
By \propref{sur not epi}, we can choose a correspondence $\pre AX_B$ such that $B_X=B$ but $X$ is not an epimorphism in the enchilada category.
Since $B_X=B$, the image of $X$ is the correspondence $\pre BB_B$.
Thus $X$ factors through its image as follows:
\[
X\cong X\otimes_B \pre BB_B.
\]
Therefore, by \cite[Proposition~I.10.1]{mitchellcategories},
if the enchilada category had equalizers then $X$ would have to be an epimorphism.
\end{proof}

Despite having both kernels and cokernels,
the enchilada category is not abelian.
In fact:

\begin{corollary}\label{not additive}
The enchilada category is not additive.
\end{corollary}

\begin{proof}
Every additive category with kernels has equalizers.
\end{proof}

Direct sum is a binary operation on isomorphism classes of $A-B$ correspondences. However,  $([\pre AX_B],\bigoplus)$ does not have a group structure since direct sum is not cancellative (i.e., we can have $X\oplus Y\cong X\oplus Z$ but $Y\not\cong Z$). Therefore, the enchilada category is not even preadditive.

\propref{split mono is  left-full} and  \propref{hb} have a dual counterparts:

\begin{prop}\label{split epi}
An $A-B$ correspondence $X$ is a right-full \hbm\ if and only if $X$ is a split epimorphism in the enchilada category.
\end{prop}

\begin{proof}
If $X$ is a right-full \hbm,
then, similarly to the proof of Proposition~\ref{split mono is left-full},
$\wilde X$ is a right inverse of $X$,
so $X$ is a split epimorphism.
\end{proof}
 
 \begin{prop}
The Enchilada category has a split epimorphism that is not the isomorphism class of a Hilbert bimodule. 
\end{prop}

\begin{proof}

Let $\KK=\KK(H)$ for an infinite-dimensional Hilbert space $H$,
and let $\wilde\KK$ be the (minimal) unitization,
so in this case $\KK+\CC 1_H$.
Then we have a short exact sequence of \csta s:
\[
\begin{tikzcd}
0 \arrow[r]
&\KK \arrow[r,hook]
&\wilde\KK \arrow[r,"q",twoheadrightarrow]
&\CC \arrow[r]
&0.
\end{tikzcd}
\]
Let $X$ be the $\wilde\KK-\CC$ correspondence given by the quotient map $q$ in the usual way.
We first show that $X$ is a split epimorphism.
Let $Y$ be the $\CC-\wilde\KK$ correspondence given by the inclusion map of the nondegenerate \cst-subalgebra $\CC$ of $\wilde\KK$.
Let
\[
\Phi:Y\od X\to \CC
\]
be the unique linear map associated with the bilinear map
\[
(y,\lambda)\mapsto q(y)\lambda.
\]
If we can verify that $\Phi$ preserves inner products and left $\CC$-module actions, then we will be able to conclude that it canonically determines an isomorphism
$Y\xt_{\wilde\KK}X\simeq \CC$ of $\CC-\CC$ correspondences:
\begin{align*}
\<\Phi(y\xt\lambda),\Phi(z\xt\mu)\>_\CC
&=\<q(y)\lambda,q(z)\mu\>_\CC
\\&=\bar\lambda\,\overline{q(y)}q(z)\mu
\\&=\bar\lambda q(y^*z)\mu
\\&=\<\lambda,\<y,z\>_{\wilde\KK}\cdot \mu\>_\CC
\\&=\<y\xt\lambda,z\xt\mu\>_\CC,
\end{align*}
and
of course $\Phi$ preserves left $\CC$-module actions
because $q$ is linear.

On the other hand, to see that $X$ is not a Hilbert bimodule notice that
$\wilde\KK$ does not have an ideal isomorphic to $\CC$,
because the only nontrivial proper ideal is $\KK$.
(Note that as a Hilbert module, $X$ is just the standard one determined by the \csta\ $\CC$, and so $\KK(X)=\CC$.)
\end{proof}

\begin{remark}\label{strong}
In \cite{arduini}, Arduini proposed a strengthening of the concepts of monomorphism and epimorphism,
with an eye toward improving the concept of subobject.
Here we only give his definition of monomorphism,
which nowadays is called
\emph{strong monomorphism}:
it
is a morphism $f:A\to B$ such that
for every commutative diagram
\[
\xymatrix{
C \ar[r]^-g \ar[d]_h
&D \ar[d]^k
\\
A \ar[r]_-f
&B,
}
\]
where $g$ is an epimorphism,
there is a unique morphism $\ell$
making the diagram
\[
\xymatrix{
C \ar[r]^-g \ar[d]_h
&D \ar[d]^k \ar@{-->}[dl]^\ell_{!}
\\
A \ar[r]_-f
&B
}
\]
commute.
Every split monomorphism is a strong monomorphism, and
every strong monomorphism is a monomorphism.
Thus, the uniqueness of $\ell$ is automatic.
Moreover, since $g$ is an epimorphism, it is enough to know that the upper triangle commutes.
As we mentioned in the introduction, for a time we thought we would be able to use Arduini's strong monomorphisms to define the image of a morphism in the enchilada category. But we had to abandon this approach, since we have an inadequate understanding of epimorphisms in the enchilada category.
\end{remark}

\end{document}